\documentclass[a4paper,10pt]{article}
\usepackage[latin1]{inputenc}
\usepackage[T1]{fontenc}
\usepackage{amsmath}
\usepackage{amsfonts}
\usepackage{amstext}
\usepackage{amsthm}
\usepackage[all]{xy}
\usepackage{geometry}
\usepackage{graphicx}
\newtheorem{theorem}{Theorem}[section]
\newtheorem{lemma}[theorem]{Lemma}
\newtheorem{defi}[theorem]{Definition}
\newtheorem{rem}[theorem]{Remark}
\newtheorem{prop}[theorem]{Proposition}
\newtheorem{cor}[theorem]{Corollary}

\DeclareMathOperator{\im}{im}
\DeclareMathOperator{\ind}{ind}

\DeclareMathOperator{\coker}{coker}

\DeclareMathOperator{\pr}{pr}

\title{The index bundle for Fredholm morphisms}
\author{Nils Waterstraat}

\begin{document}
\date{}
\maketitle

\footnotetext[2]{The author was supported by the research training group 1493~``Mathematical Structures in Modern Quantum Physics''~of the German Research Foundation (DFG) and by the VIGONI program of the German Academic Exchange Service (DAAD).}

\begin{abstract}
We extend the index bundle construction for families of bounded Fredholm operators to morphisms between Banach bundles.
\end{abstract}

\section{Introduction}
In the sixties Atiyah and J\"anich constructed independently a map

\begin{align}\label{indmap}
\ind:[X,\mathcal{BF}(H)]\rightarrow K(X),
\end{align}
which assigns to every homotopy class of families of bounded Fredholm operators acting on the separable Hilbert space $H$ a virtual bundle (cf. \cite[App. A]{KTheoryAtiyah}). This so called \textit{index bundle} shares several formal properties with the ordinary Fredholm index and, moreover, \eqref{indmap} turned out to be an isomorphism which is nowadays known as the Atiyah-J\"anich theorem. The proof of this result shows the exactness of the sequence of semi-groups

\begin{align}\label{exactsequence}
1\rightarrow[X,GL(H)]\xrightarrow{\iota}[X,\mathcal{BF}(H)]\xrightarrow{\ind} K(X)\rightarrow 1,
\end{align}
which implies the assertion since $GL(H)$ is contractible according to Kuiper's theorem.\\
Later the construction of the index bundle was extended to families of Fredholm operators acting on a general Banach space $E$ and the exactness of the corresponding sequence \eqref{exactsequence} was shown under additional assumptions on $E$ (see \cite[\S 2]{Banach Bundles} for a survey). Since $GL(E)$ has been proved to be contractible for a wide class of Banach spaces, $\mathcal{BF}(E)$ represents the $K$-theory functor for many common sequence and function spaces.\\
The aim of this paper is to generalise the index bundle construction from families of Fredholm operators acting on a fixed Banach space to Fredholm morphisms $L:\mathcal{E}\rightarrow\mathcal{F}$ between Banach bundles over a compact base space $X$. Here some technical difficulties arise which are mainly caused by the target bundle $\mathcal{F}$. We solve these problems by using deep results from the theory of Banach bundles which are based on the fact that every infinite dimensional Banach bundle admits a nowhere vanishing continuous section. Besides the elementary properties of our index bundle $\ind L\in K(X)$, we also show the exactness at the element in the middle of a corresponding sequence like \eqref{exactsequence} and discuss an application.\\
The paper is structured as follows: In the first section we recall some well known generalities on Banach bundles and use them to recall the classical index bundle construction for families of Fredholm operators acting on a fixed Banach space. Afterwards we treat more advanced theory about Banach bundles where we give complete and detailed proofs although the results are well known. In the second section we use these preliminaries in order to prove two quite technical theorems that lead eventually to the definition of the index bundle for Fredholm morphisms. In the third section we discuss its elementary properties and show that we obtain the classical definition if $\mathcal{E}=\mathcal{F}=X\times E$ for some Banach space $E$. In the fourth section we consider the corresponding sequence \eqref{exactsequence}. Finally, in the fifth section we give an application to families of boundary value problems of first order ordinary differential operators.\\
Since all constructions hold for real and complex Banach bundles, we write $\mathbb{K}$ for the underlying scalars. However, in order to simplify notation, we denote by $K(X)$ the complex and real $K$-theory of the compact topological space $X$. Accordingly, the reader should keep in mind that one has to substitute $K(X)$ by $KO(X)$ if real Banach bundles are considered. Finally, we write $\Theta(V)$ for the trivial vector bundle with fibre $V$ over $X$.

\section{Some generalities on Banach bundles and their morphisms}\label{BanachBundles}

\subsection{Elementary properties and the index bundle}

In the following we denote the total spaces of Banach bundles by calligraphic letters as $\mathcal{E}$, $\mathcal{F}$ and their model spaces by the corresponding latin letters $E,F$ without further reference. We denote by $\mathcal{L}(\mathcal{E},\mathcal{F})$ the space of all morphisms and abbreviate notation by $\mathcal{L}(\mathcal{E}):=\mathcal{L}(\mathcal{E},\mathcal{E})$ as usual. \\
We now recall some well known results which will be used in the following sections without further reference. Detailed proofs can be found for example in \cite{ThesisIch}. We begin with the close relation between projections and subbundles. Recall that a subbundle $\mathcal{F}$ of a Banach bundle $\mathcal{E}$ is called direct if each fibre $\mathcal{F}_\lambda$ is a complemented subspace of the corresponding fibre $\mathcal{E}_\lambda$.

\begin{prop}\label{BanachBundles-prop-projections}
Let $\mathcal{E}$ be a Banach bundle and let $P\in\mathcal{L}(\mathcal{E})$ be an idempotent morphism. Then

\begin{align*}
\im(P)=\{u\in\mathcal{E}:Pu=u\}\subset\mathcal{E}
\end{align*}
is a direct subbundle of $\mathcal{E}$.\\
Conversely, if the base space is paracompact and $\mathcal{F}\subset\mathcal{E}$ is a direct subbundle of $\mathcal{E}$, then there exists an idempotent morphism $P\in\mathcal{L}(\mathcal{E})$ such that $\mathcal{F}=\im(P)$.
\end{prop}

Moreover, the following perfect analogy with subspaces of Banach spaces holds.

\begin{prop}\label{BanachBundles-prop-splittingmor}
Let $\mathcal{E}$ be a Banach bundle over a paracompact base space and let $\mathcal{F},\mathcal{G}$ be subbundles of $\mathcal{E}$. Then $\mathcal{E}\cong\mathcal{F}\oplus\mathcal{G}$ if and only if there exists an idempotent morphism $P\in\mathcal{L}(\mathcal{E})$ such that $\im(P)=\mathcal{F}$ and $\ker(P)=\mathcal{G}$.
\end{prop}

\begin{cor}\label{BanachBundles-cor-splitting}
Let $\mathcal{E}$ be a Banach bundle over a paracompact base and let $\mathcal{F}$ be a direct subbundle. Then $\mathcal{F}$ is complemented, that is,  there exists a direct subbundle $\mathcal{G}\subset\mathcal{E}$, such that

\begin{align*}
\mathcal{E}\cong\mathcal{F}\oplus\mathcal{G}.
\end{align*}

\end{cor}

Finally, we need that kernels and images of bundle morphisms can give rise to subbundles of the corresponding target and domain bundles. 

\begin{lemma}\label{BanachBundles-cor-kerbundle}
Let $\mathcal{E},\mathcal{F}$ be Banach bundles over the paracompact base $X$ and let $L\in\mathcal{L}(\mathcal{E},\mathcal{F})$ be a morphism.

\begin{enumerate}
	\item[(i)] If the kernels $\ker L_\lambda$, $\lambda\in X$, form a direct subbundle of $\mathcal{E}$ and each $L_\lambda$ has a complemented image, then the spaces $\im L_\lambda$ form a direct subbundle of $\mathcal{F}$.
	\item[(ii)] If the images $\im L_\lambda$ form a direct subbundle of $\mathcal{F}$ and each $L_\lambda$ has a complemented kernel, then the spaces $\ker L_\lambda$ form a direct subbundle of $\mathcal{E}$. 
\end{enumerate}
\end{lemma}

We now briefly sketch the classical construction of the index bundle. Accordingly, let $E$ be a Banach space, let $X$ be a compact topological space and let $L:X\rightarrow\mathcal{BF}(E)$ be a family of bounded Fredholm operators which we assume to be continuous with respect to the norm topology on $\mathcal{BF}(E)$. The main observation for defining the index bundle is that there exists a closed subspace $E_1\subset E$ of finite codimension such that $\ker L_\lambda\cap E_1=\{0\}$ for all $\lambda\in X$ (cf. \cite[Lemma 2.1]{Banach Bundles}). Then the restrictions $L_\lambda\mid_{E_1}$ to $E_1$ define a bundle monomorphism of $\Theta(E_1)$ into $\Theta(E)$. Since each $\im L_\lambda\mid_{E_1}$ is complemented, $\im L\mid_{X\times E_1}$ defines a subbundle of $\Theta(E)$ according to lemma \ref{BanachBundles-cor-kerbundle}. Consequently, by proposition \ref{BanachBundles-prop-projections} we can find a family of projections $P:X\rightarrow\mathcal{L}(E)$ such that $\im P_\lambda=\im L_\lambda\mid_{E_1}$, $\lambda\in X$. Now the index bundle of $L$ is defined by

\begin{align*}
\ind L=[\Theta(E/{E_1})]-[\im(I-P)]\in K(X).
\end{align*}
Its well definedness and main properties can easily be worked out by the reader but will also follow from our results in sections 3 and 4.


\subsection{Sections and finite dimensional subbundles}
In this section we consider finite dimensional subbundles of infinite dimensional Banach bundles, where results can be obtained that are in marked contrast to the case of finite dimensional bundles. The reason for the appearance of unexpected phenomena is given by the deep theorem that every infinite dimensional Banach space $E$ is homeomorphic to $E\setminus\{0\}$.\\
Our main reference here is \cite{Banach Bundles}, however, we use a slightly different presentation and give more details at points that will become important for us in the construction of the index bundle below.

\begin{theorem}\label{BanachBundles-theorem-nowhere}
Let $p:\mathcal{E}\rightarrow X$ be an infinite dimensional Banach bundle over the paracompact  base $X$ and let $\{\delta_1,\ldots,\delta_m\}$ be pointwise linear independent sections of $\mathcal{E}$ over an open set $U\subset X$. Then for every closed $D\subset U$ there exist pointwise linear independent sections $\{\tilde{\delta}_1,\ldots,\tilde{\delta}_m\}$ of the whole bundle $\mathcal{E}$ such that $\tilde{\delta}_i\mid_D=\delta_i$, $i=1,\ldots,m$. 
\end{theorem}

\begin{proof}
At first we consider the special case $m=1$. Let $\{U_k\}_{k\in\mathbb{N}}$ be a locally finite countable cover of $X$ by trivialising neighbourhoods which exists for any fibre bundle over a paracompact base according to \cite[5.9]{MiSta}. We define $V_0:=U$ and $V_{k}=U_k\cap(X\setminus D)$, $k\in\mathbb{N}$, and obtain a locally finite countable open cover $\{V_k\}_{k\in\mathbb{N}\cup\{0\}}$ of $X$ such that $V_k\cap D=\emptyset$, $k\in\mathbb{N}$. Moreover, we choose for each $k\in\mathbb{N}$ a trivialisation $\varphi_k:p^{-1}(V_{k})\rightarrow V_{k}\times E$ of $\mathcal{E}$.\\
We now construct the required section by induction and show that whenever we have a nowhere vanishing section $\delta_n$ on 

\begin{align*}
W_n:=\bigcup_{0\leq k\leq n}{V_k},\quad n\in\mathbb{N}\cup\{0\},
\end{align*}
we can find a nowhere vanishing section $\delta_{n+1}$ on $W_{n+1}$ such that $\delta_{n+1}\mid_ D=\delta_n\mid_D$.  Since on $W_0=U$ we have by assumption a nowhere vanishing section that coincides with itself over $D$, this shows the existence of the desired section on each $W_n$, $n\in\mathbb{N}$. We will explain below why our construction actually yields a continuous section on all of $X$.\\
Since $E\setminus\{0\}$ is homeomorphic to $E$, there exists a linear structure $\{\hat{+},\hat{\cdot}\}$ on $E\setminus\{0\}$ that is compatible with the given topology on this space. Let now $\delta_n$ be constructed over $W_n$. We take a partition of unity $\{\eta_{1,n},\eta_{2,n}\}$ subordinated to the open covering $\{W_n,V_{n+1}\}$ of $W_{n+1}$ and define 

\begin{align*}	\delta'_{n+1}(\lambda)=\varphi^{-1}_{n+1}(\eta_{1,n}(\lambda)\hat{\cdot}\varphi_{n+1}(\delta_n(\lambda))\hat{+}\eta_{2,n}(\lambda)\hat{\cdot} u),\quad \lambda\in V_{n+1},
\end{align*} 
for some $u\in E\setminus\{0\}$. Then

\begin{align*}
\delta_{n+1}(\lambda)=\begin{cases}
\delta_n(\lambda),\quad \lambda\notin V_{n+1}\\
\delta'_{n+1}(\lambda),\quad \lambda\in V_{n+1} 
\end{cases}
\end{align*}
is a nowhere vanishing continuous section over $W_{n+1}$. Moreover, $\delta_{n+1}$ coincides with $\delta_1$ on $D$ since $V_k\cap D=\emptyset$ for all $k\in\mathbb{N}$.\\  
By the local finiteness of $\{V_k\}_{k\in\mathbb{N}}$, the construction is such that for any $\lambda_0\in X$ there exists an $n_0\in\mathbb{N}$ and a neighbourhood $U_{\lambda_0}$ of $\lambda_0$ such that $U_{\lambda_0}\cap V_n=\emptyset$ for all $n\geq n_0$. Accordingly, $\delta_n\mid_{U_{\lambda_0}}=\delta_{n_0}\mid_{U_{\lambda_0}}$ for all $n\geq n_0$. Hence any point $\lambda_0$ has a neighbourhood on which the construction of $\delta$ is finished after a finite number of steps and so we indeed obtain a nowhere vanishing continuous section.\\
Now we turn to the case $m>1$. By the special case already proven we can extend $\delta_1$ to a nowhere vanishing section $\tilde{\delta}_1$ of $\mathcal{E}$. Then $\tilde{\delta}_1$ defines a one dimensional subbundle $\mathcal{F}_1$ of $\mathcal{E}$. Moreover, the sections $\{\delta_2,\ldots,\delta_m\}$ build a subbundle of $\mathcal{E}$ restricted to $U$. Now we choose an open subset $U_1\subset X$ such that $D\subset U_1\subset\overline{U}_1\subset U$. It is not hard to construct a projection $P\in\mathcal{L}(\mathcal{E})$ having the one dimensional subbundle $\mathcal{F}_1$ of $\mathcal{E}$ as image and such that the remaining sections $\delta_2,\ldots,\delta_m$ are in its kernel over $U_1$. Hence we obtain a splitting $\mathcal{E}=\mathcal{F}_1\oplus\mathcal{G}_1$ such that $\delta_1$ is a section of the one dimensional bundle $\mathcal{F}_1$ and $\delta_2,\ldots,\delta_m$ are sections over $U_1$ of the infinite dimensional Banach bundle $\mathcal{G}_1$. Now we can apply the first part of our proof to $\delta_2$ as a section of $\mathcal{G}_1$ over $U_1$. Iterating this construction we obtain the desired sections $\tilde{\delta}_1,\ldots,\tilde{\delta}_m:X\rightarrow\mathcal{E}$. 
\end{proof} 

We note a couple of immediate, but important consequences.

\begin{cor}\label{BanachBundles-cor-nowhere}
Let $p:\mathcal{E}\rightarrow X$ be a Banach bundle over a paracompact base such that $\dim E=\infty$. Then

\begin{enumerate}
\item[(i)] There exists a nowhere vanishing section of $\mathcal{E}$.
\item[(ii)] For any $n\in\mathbb{N}$, $\mathcal{E}$ contains an $n$-dimensional trivial subbundle.
\item[(iii)] If $X$ is compact and $\mathcal{F}$ is a finite dimensional vector bundle over $X$, then $\mathcal{E}$ contains a subbundle isomorphic to $\mathcal{F}$.
\item[(iv)] If $X$ is compact and $\mathcal{F}$ is a finite dimensional subbundle of $\mathcal{E}$, then there exists a further finite dimensional subbundle $\mathcal{F}'$ of $\mathcal{E}$ such that $\mathcal{F}\oplus\mathcal{F}'$ is a trivial subbundle of $\mathcal{E}$.
\item[(v)] Let $U\subset X$ be open and $\mathcal{F}$ a finite dimensional, trivial subbundle of $\mathcal{E}$ over $U$. Then for any closed subset $D\subset U$ there exists a finite dimensional trivial subbundle of $\mathcal{E}$ that coincides with $\mathcal{F}$ over $D$.
\end{enumerate}

\end{cor}

\begin{cor}\label{BanachBundles-cor-bundlesplitting}
Let $\mathcal{E}$ be an infinite dimensional Banach bundle over the  paracompact base $X$, let $U\subset X$ be an open subset and let $\mathcal{F}_1,\mathcal{F}_2$ be subbundles of $\mathcal{E}\mid_U$ such that $\mathcal{E}\mid_U=\mathcal{F}_1\oplus\mathcal{F}_2$, where $\mathcal{F}_1$ is finite dimensional and trivial. Then for each closed set $D\subset U$ there is a decomposition $\mathcal{E}=\mathcal{E}_1\oplus\mathcal{E}_2$, where $\mathcal{E}_1$ is finite dimensional and trivial, such that $\mathcal{E}_i\mid_D=\mathcal{F}_i\mid_D$, $i=1,2$.
\end{cor}

\begin{proof}
Let $U_1\subset X$ be an open subset such that $D\subset U_1\subset\overline{U_1}\subset U$. 
By the last part of corollary \ref{BanachBundles-cor-nowhere} there is a finite dimensional trivial subbundle $\mathcal{E}_1$ which coincides with $\mathcal{F}_1$ on $\overline{U_1}$. We choose a partition of unity $\{\eta_1,\eta_2\}$ corresponding to the open cover $\{U_1,X\setminus D\}$ of $X$. Furthermore, let $P\in\mathcal{L}(\mathcal{E})$ be a projection on the subbundle $\mathcal{E}_1$ and $Q\in\mathcal{L}(\mathcal{E}\mid_{\overline{U_1}})$ a projection on $\mathcal{E}_1\mid_{\overline{U_1}}\,=\mathcal{F}_1\mid_{\overline{U_1}}$ parallel to $\mathcal{F}_2\mid_{\overline{U_1}}$. Then the morphism $\eta_1 Q+\eta_2 P$ is a projection onto $\mathcal{E}_1$ and its kernel gives a subbundle $\mathcal{E}_2$ such that $\mathcal{E}=\mathcal{E}_1\oplus\mathcal{E}_2$. Moreover, $\mathcal{E}_2$ coincides with $\mathcal{F}_2$ over $D$ by construction. 
\end{proof}

We denote in the following by $\pi:X\times[0,1]\rightarrow X$ the projection onto the first component and by 

\begin{align*}
\iota_0:X\hookrightarrow X\times\{0\}\subset X\times[0,1],\quad \iota_1:X\hookrightarrow X\times\{1\}\subset X\times[0,1] 
\end{align*}
the canonical inclusions.

\begin{defi}
Two subbundles $\mathcal{F}_0$ and $\mathcal{F}_1$ of a Banach bundle $\mathcal{E}$ are said to be homotopic if there exists a subbundle $\mathcal{F}$ of $\pi^\ast\mathcal{E}$ such that

\begin{align*}
\iota^\ast_i\mathcal{F}=\mathcal{F}_i,\quad i=0,1.
\end{align*}
\end{defi} 

\begin{rem}\label{BanachBundles-rem-homiso}
As in the case of finite fibre dimensions one can show that homotopic bundles over a paracompact base space are isomorphic.
\end{rem}

Now we can state a further corollary of theorem \ref{BanachBundles-theorem-nowhere}.

\begin{cor}\label{BanachBundles-cor-homotopicI}
Two trivial finite dimensional subbundles of the same dimension of an infinite dimensional Banach bundle over a paracompact base are homotopic.
\end{cor}

\begin{proof}
We denote by $\pi_1:X\times[0,\frac{1}{3})\rightarrow X$ and $\pi_2:X\times(\frac{2}{3},1]\rightarrow X$ the projections on the first component. If $\mathcal{F}_1$ and $\mathcal{F}_2$ are two trivial finite dimensional bundles over $X$, we obtain a finite dimensional trivial bundle over $X\times([0,\frac{1}{3})\cup(\frac{2}{3},1])$ by $\pi^\ast_1\mathcal{F}_1\cup\pi^\ast_2\mathcal{F}_2$. Setting $D=X\times(\{0\}\cup\{1\})$, the claim follows from the last part of corollary \ref{BanachBundles-cor-nowhere}.
\end{proof}

The next goal is to improve this result and to obtain some interesting consequences that will become important in the construction of the index bundle below.

\begin{prop}\label{BanachBundles-prop-homotopic}
Let $\mathcal{E}$ be an infinite dimensional Banach bundle over the compact base $X$ and let $\mathcal{F}_0,\mathcal{F}_1$ be two isomorphic finite dimensional subbundles. Then $\mathcal{F}_0$ and $\mathcal{F}_1$ are homotopic.
\end{prop}

\begin{proof}
Since $\mathcal{F}_0$ and $\mathcal{F}_1$ are isomorphic, we can find a finite dimensional bundle $\mathcal{M}$ over $X$ such that $\mathcal{F}_0\oplus\mathcal{M}$ and $\mathcal{F}_1\oplus\mathcal{M}$ are trivial vector bundles of some dimension $n\in\mathbb{N}$. Now, by using theorem \ref{BanachBundles-theorem-nowhere}, we can embed these bundles into $\mathcal{E}$ and obtain that $\mathcal{F}_0$ and $\mathcal{F}_1$ are contained in trivial subbundles $\mathcal{G}_0$ and $\mathcal{G}_1$ of $\mathcal{E}$, respectively, which are of the same finite dimension $n$. Then $\mathcal{G}_0$ and $\mathcal{G}_1$ are homotopic by corollary \ref{BanachBundles-cor-homotopicI} and, moreover, since $\mathcal{G}_0$ and $\mathcal{G}_1$ are trivial, the corresponding subbundle $\mathcal{G}$ of $\pi^\ast\mathcal{E}$ is trivial as well. Now we use a complement of $\mathcal{G}$ in $\pi^\ast\mathcal{E}$ and part (ii) of corollary \ref{BanachBundles-cor-nowhere} in order to extend $\mathcal{G}$ to a $2n$-dimensional trivial subbundle $\tilde{\mathcal{G}}$ of $\pi^\ast\mathcal{E}$. Since we can embed $\mathcal{F}_0$ and $\mathcal{F}_1$ into $\Theta(\mathbb{K}^n)$, there exist continuous maps $f_0,f_1:X\rightarrow G_k(\mathbb{K}^n)$, $k=\dim\mathcal{F}_i$, such that $f^\ast_i\gamma_k(\mathbb{K}^n)\cong\mathcal{F}_i$, $i=0,1$ (cf. \cite[3.5.3]{Husemoller}). Then we obtain from \cite[3.6.2]{Husemoller} and the assumption that $\mathcal{F}_0$ and $\mathcal{F}_1$ are isomorphic that $j\circ f_0$ and $j\circ f_1$ are homotopic, where $j:G_k(\mathbb{K}^n)\hookrightarrow G_k(\mathbb{K}^{2n})$ denotes the canonical inclusion. Now it is easily seen that $\mathcal{F}_0$ and $\mathcal{F}_1$ are homotopic by a subbundle of $\tilde{\mathcal{G}}$.   
\end{proof}

\begin{cor}\label{BanachBundles-cor-homotopicII}
Let $\mathcal{E}$ be an infinite dimensional Banach bundle over the compact base $X$ and let $\mathcal{F}_0,\mathcal{F}_1$ be two isomorphic finite dimensional subbundles. Then any complements of $\mathcal{F}_0$ and $\mathcal{F}_1$ are homotopic as well and in particular isomorphic.
\end{cor}

\begin{proof}
Let $\mathcal{F}'_0$ and $\mathcal{F}'_1$ denote complements to $\mathcal{F}_0$ and $\mathcal{F}_1$, respectively. By proposition \ref{BanachBundles-prop-homotopic}, $\mathcal{F}_0$ and $\mathcal{F}_1$ are homotopic. It is easy to construct an idempotent $P\in\mathcal{L}(\pi^\ast\mathcal{E})$ having the homotopy between $\mathcal{F}_0$ and $\mathcal{F}_1$ as image and such that $\mathcal{F}'_0$ and $\mathcal{F}'_1$ are in the kernel of $P$ over $X\times\{0\}\cup X\times\{1\}$. Then the image of $I-P$ gives a homotopy of the complements. 
\end{proof}

\begin{cor}\label{BanachBundles-cor-comptriv}
Let $\mathcal{E}$ be a trivial infinite dimensional Banach bundle over the compact base $X$ and let $\mathcal{F}$ be a trivial finite dimensional subbundle. Then any complement to $\mathcal{F}$ is trivial as well.
\end{cor}

\begin{proof}
Let $\varphi:\mathcal{E}\rightarrow X\times E$ be a global trivialisation of $\mathcal{E}$ and let $F\subset E$ be a subspace having the same finite dimension than $\mathcal{F}$. Then the counterimage of $F$ under $\varphi$ defines a trivial subbundle $\mathcal{F}'$ of $\mathcal{E}$. Moreover, by taking any complementary subspace of $F$ in $E$ we obtain in the same way a trivial subbundle of $\mathcal{E}$ which is complementary to $\mathcal{F}'$. Now $\mathcal{F}$ and $\mathcal{F}'$ are homotopic by corollary \ref{BanachBundles-cor-homotopicI} and so any subbundle complementary to $\mathcal{F}$ is trivial by corollary \ref{BanachBundles-cor-homotopicII}. 
\end{proof}


\section{The construction of the index bundle}

If $\mathcal{E}$ and $\mathcal{F}$ are Banach bundles over $X$, we denote by $\mathfrak{F}_k(\mathcal{E},\mathcal{F})$, $k\in\mathbb{Z}$, the space of all morphisms between $\mathcal{E}$ and $\mathcal{F}$ that are bounded Fredholm operators of index $k$ in each fibre and by $\mathfrak{F}(\mathcal{E},\mathcal{F})$ the union of all these spaces. As in the foregoing section we will use the convention to denote the model spaces of $\mathcal{E}$ and $\mathcal{F}$ by $E$ and $F$, respectively.\\
Now we begin the construction of the index bundle and require in the following that the base space $X$ of the Banach bundles $\mathcal{E}$ and $\mathcal{F}$ is compact. Let $L\in\mathfrak{F}(\mathcal{E},\mathcal{F})$ be a Fredholm morphism and assume that $\mathcal{V}\subset\mathcal{F}$ is a finite dimensional subbundle that is transversal to $\im(L)$ in the sense that

\begin{align}\label{FredMor-align-transversal}
\im(L_\lambda)+\mathcal{V}_\lambda=\mathcal{F}_\lambda\quad\text{for all}\;\lambda\in X.
\end{align}
We choose a fibrewise projection $P\in\mathcal{L}(\mathcal{F})$ such that $\im(P)=\mathcal{V}$ and obtain a direct subbundle $\im(I-P)$ of $\mathcal{F}$. By the property \eqref{FredMor-align-transversal} of $\mathcal{V}$, the composition

\begin{align*}
\mathcal{E}\xrightarrow{L}\mathcal{F}\xrightarrow{I-P}\im(I-P)
\end{align*}
is a surjective bundle morphism and hence its kernel gives rise to a direct subbundle $E(L,\mathcal{V})$ of $\mathcal{E}$. Note that the total space of $E(L,\mathcal{V})$ is given by

\begin{align*}
\bigcup_{\lambda\in X}{\{u\in\mathcal{E}_\lambda:L_\lambda u\in\mathcal{V}_\lambda\}}
\end{align*}
and its fibres are of dimension

\begin{align}\label{FredMor-lemma-dimE}
\dim E(L,\mathcal{V})_\lambda=\ind(L_\lambda)+\dim\mathcal{V}_\lambda,\quad\lambda\in X.
\end{align}
We want to define the index bundle of the morphism $L$ by

\begin{align}\label{FredMor-align-indbundle}
[E(L,\mathcal{V})]-[\mathcal{V}]\in K(X)
\end{align}
but we have to make sure before that finite dimensional bundles $\mathcal{V}$ as in \eqref{FredMor-align-transversal} indeed exist and that the element \eqref{FredMor-align-indbundle} is independent of their choice.

\begin{theorem}\label{FredMor-theorem-transversalexistence}
Let $L\in\mathfrak{F}(\mathcal{E},\mathcal{F})$ be a Fredholm morphism acting between the Banach bundles $\mathcal{E}$ and $\mathcal{F}$ over the compact base space $X$. Then there exists a finite dimensional trivial subbundle $\mathcal{V}\subset\mathcal{F}$ such that \eqref{FredMor-align-transversal} holds over $X$.
\end{theorem}

\begin{proof}
We assume in a first step that $\mathcal{F}$ is trivialised by a global trivialisation $\psi$.\\
Let $\lambda_0\in X$ and let $U_{\lambda_0}\subset X$ be an open neighbourhood of $\lambda_0$ such that $\mathcal{E}$ is trivial on $U_{\lambda_0}$ by means of a trivialisation $\varphi$. Let

\begin{align*}
\tilde{L}=\pr_2\circ\psi\circ L\circ\varphi^{-1}:U_{\lambda_0}\times E\rightarrow F
\end{align*}
denote the corresponding family of bounded Fredholm operators with respect to these trivialisations. Since $\tilde{L}_{\lambda_0}$ is Fredholm, there exists $V_{\lambda_0}\subset F$, $\dim V_{\lambda_0}<\infty$, and $W_{\lambda_0}\subset E$ closed such that $\im(\tilde L_{\lambda_0})\oplus V_{\lambda_0}=F$ and $\ker(\tilde L_{\lambda_0})\oplus W_{\lambda_0}=E$. Now consider

\begin{align*}
A_{\lambda}:W_{\lambda_0}\times V_{\lambda_0}\rightarrow F,\;\; A_{\lambda}(w,v)=\tilde L_\lambda w+v.
\end{align*}
Because of $A_{\lambda_0}\in GL(W_{\lambda_0}\times V_{\lambda_0},F)$ and the continuity of $A:U_{\lambda_0}\rightarrow \mathcal{L}(W_{\lambda_0}\times V_{\lambda_0},F)$, there exists a neighbourhood $\tilde U_{\lambda_0}\subset U_{\lambda_0}$ of $\lambda_0$ such that $A_\lambda\in GL(W_{\lambda_0}\times V_{\lambda_0},F)$ and hence

\begin{align*}
\im(\tilde L_\lambda)+V_{\lambda_0}=F\;\;\text{ for all }\; \lambda\in \tilde U_{\lambda_0}.
\end{align*}
By compactness we can now cover $X$ by a finite number of neighbourhoods $\tilde{U}_{\lambda_i}$, $i=1,\ldots,n$, such that for each $i$ there exists a finite dimensional subspace $V_{\lambda_i}$ such that

\begin{align*}
\im(\pr_2\circ\psi_\lambda\circ L_\lambda)+V_{\lambda_i}=F\;\;\text{for all}\,\, \lambda\in\tilde{U}_{\lambda_i},\quad i=1,\ldots,n.
\end{align*} 
Finally, $V:=V_1+\ldots+V_n$ defines a finite dimensional subspace of $F$ such that 

\begin{align*}
\im(\pr_2\circ\psi_\lambda\circ L_\lambda)+V=F\;\;\text{for all}\,\, \lambda\in X.
\end{align*} 
Then $\psi^{-1}(X\times V)$ is a finite dimensional trivial subbundle of $\mathcal{F}$ such that \eqref{FredMor-align-transversal} holds on all of $X$ and the assertion is proved in the special case that $\mathcal{F}$ is trivial.\\
We now turn to the general case. Let $U^0_k$, $k=1,\ldots,N$, be a finite open covering of $X$ such that $\mathcal{F}$ is trivial over each $U^0_k$. Moreover, we choose open sets $U^i_k$, $i=1,2$, $k=1,\ldots, N$, such that $\overline{U^{i+1}_k}\subset U^i_k$, $i=0,1$, and $\{U^2_k\}$ is still an open covering of $X$.\\
Consider $U^0_1$. Since $\mathcal{F}$ is trivial on the compact subspace $\overline{U^1_1}$, we obtain from the special case in which we proved the assertion already that there is a finite dimensional trivial subbundle $\mathcal{V}'$ of $\mathcal{F}$ over $\overline{U^1_1}$ such that \eqref{FredMor-align-transversal} holds. We choose a complement $\mathcal{W}'$ to $\mathcal{V}'$ in $\mathcal{F}\mid_{\overline{U^1_1}}$. By corollary \ref{BanachBundles-cor-bundlesplitting} there exist bundles $\mathcal{V}_1$ and $\mathcal{W}_1$ over $X$ such that $\mathcal{F}=\mathcal{V}_1\oplus\mathcal{W}_1$, $\mathcal{V}_1$ is finite dimensional, trivial and

\begin{align*}
\mathcal{V}_1\mid_{\overline{U^2_1}}\,=\mathcal{V}'\mid_{\overline{U^2_1}},\qquad  \mathcal{W}_1\mid_{\overline{U^2_1}}\,=\mathcal{W}'\mid_{\overline{U^2_1}}.
\end{align*}
Note that $\mathcal{V}_1$ is a trivial bundle over all of $X$ satisfying \eqref{FredMor-align-transversal} on $U^2_1$. Moreover, since $\mathcal{F}$ is trivial on $\overline{U^1_k}$, $k=1,\ldots N$, we obtain by corollary \ref{BanachBundles-cor-comptriv} that the bundle $\mathcal{W}_1$ is trivial on all $\overline{U^1_k}$ as well.\\ 
Next we consider $U^0_2$ and let $P\in\mathcal{L}(\mathcal{F})$ denote a projection onto $\mathcal{W}_1$. Consider $P\circ L:\mathcal{E}\rightarrow\mathcal{W}_1$ which is again a Fredholm morphism. Since $\mathcal{W}_1$ is trivial on $\overline{U^1_2}$ we can argue as on $U^0_1$ above and obtain a decomposition $\mathcal{W}_1=\mathcal{V}_2\oplus\mathcal{W}_2$ such that \eqref{FredMor-align-transversal} holds for $P\circ L$ and $\mathcal{V}_2$ on $U^2_2$. Moreover, $\mathcal{V}_2$ is trivial on $X$ and $\mathcal{W}_2$ is trivial on each $\overline{U^1_k}$, $k=1,\ldots,N$. Finally, note that $\mathcal{V}_1\oplus\mathcal{V}_2$ is a subbundle of $\mathcal{F}$ which is transversal to $\im L$ over $U^2_1\cup U^2_2$.\\
Continuing this process we eventually arrive at a finite dimensional trivial subbundle $\mathcal{V}=\bigoplus^N_{i=1}{\mathcal{V}_i}$ of $\mathcal{F}$ over $X$ such that \eqref{FredMor-align-transversal} holds over all of $X=\bigcup^N_{k=1}{U^2_k}$.
\end{proof}

We leave it to the reader to check that for compact spaces $Y$ and continuous maps $f:Y\rightarrow X$, the pullback bundle $f^\ast\mathcal{V}$ is transversal to the image of the pullback morphism $f^\ast L:f^\ast\mathcal{E}\rightarrow f^\ast\mathcal{F}$ and we have

\begin{align}\label{BanachBundles-cor-prenaturality}
[E(f^\ast L,f^\ast\mathcal{V})]-[f^\ast\mathcal{V}]=f^\ast([E(L,\mathcal{V})]-[\mathcal{V}])\in K(X).
\end{align}
Now we can prove that the element \eqref{FredMor-align-indbundle} moreover does not depend on the particular choice of a finite dimensional subbundle of $\mathcal{F}$ which is transversal to $\im L$.

\begin{theorem}\label{FredMor-theorem-welldefined}
Let $L\in\mathfrak{F}(\mathcal{E},\mathcal{F})$ be a Fredholm morphism acting between the Banach bundles $\mathcal{E}$ and $\mathcal{F}$ over the compact base space $X$. If $\mathcal{V},\mathcal{W}\subset\mathcal{F}$ are two finite dimensional subbundles that are transversal to the image of $L$ in the sense of \eqref{FredMor-align-transversal}, then

\begin{align*}
[E(L,\mathcal{V})]-[\mathcal{V}]=[E(L,\mathcal{W})]-[\mathcal{W}]\in K(X).
\end{align*}

\end{theorem}

\begin{proof}
There are four steps.

\subsubsection*{Step 1}
We consider the case that $\mathcal{V}\subset\mathcal{W}$ is a subbundle. Then $E(L,\mathcal{V})\subset E(L,\mathcal{W})$ is a subbundle as well and hence we can find a complement bundle $E(L,\mathcal{V})^\perp$ such that $E(L,\mathcal{W})=E(L,\mathcal{V})\oplus E(L,\mathcal{V})^\perp$. Since $L\mid_{E(L,\mathcal{V})^\perp}$ is injective by the definition of $E(L,\mathcal{V})$, we conclude that $L(E(L,\mathcal{V})^\perp)\subset\mathcal{W}$ is a subbundle and the restriction of $L$ defines a bundle isomorphism $L\mid_{E(L,\mathcal{V})^\perp}:E(L,\mathcal{V})^\perp\rightarrow L(E(L,\mathcal{V})^\perp)$. Moreover, from the definition of $E(L,\mathcal{V})$ and \eqref{FredMor-lemma-dimE} it is readily seen that $\mathcal{V}\oplus L(E(L,\mathcal{V})^\perp)=\mathcal{W}$ and hence we obtain

\begin{align*}
[E(L,\mathcal{W})]-[\mathcal{W}]&=[E(L,\mathcal{V})\oplus E(L,\mathcal{V})^\perp]-[\mathcal{V}\oplus L(E(L,\mathcal{V})^\perp)]\\
&=[E(L,\mathcal{V})]-[\mathcal{V}].
\end{align*}

\subsubsection*{Step 2}
In the second step of the proof we now turn to the general case and consider two finite dimensional bundles $\mathcal{V}$, $\mathcal{W}$ as in the assertion. Our aim is to use the special case we already proved above in order to show that we can assume without loss of generality that $\mathcal{V}$ and $\mathcal{W}$ are trivial and of the same dimension. In order to do so, note at first that $\mathcal{V}$ and $\mathcal{W}$ are contained in  trivial finite dimensional subbundles of $\mathcal{F}$ by corollary \ref{BanachBundles-cor-nowhere}. Hence by the special case proved above, we can assume without loss of generality that $\mathcal{V}$ and $\mathcal{W}$ are trivial. If now $\dim\mathcal{V}=\dim\mathcal{W}$, we are done. If, however, they are not of the same dimension, say $\dim\mathcal{V}<\dim\mathcal{W}$, we choose a subbundle $\mathcal{M}$ of $\mathcal{E}$ which is complementary to $\mathcal{V}$ in $\mathcal{E}$. According to corollary \ref{BanachBundles-cor-nowhere}, we now can find a $(\dim\mathcal{W}-\dim\mathcal{V})$-dimensional trivial subbundle of $\mathcal{M}$ and the direct sum of this bundle and $\mathcal{V}$ yields a trivial subbundle of $\mathcal{E}$ of the same dimension than $\mathcal{W}$. By using the first step of our proof once again, we finally obtain that it suffices to prove the assertion of the theorem under the additional assumption that $\mathcal{V}$ and $\mathcal{W}$ are trivial and of the same dimension.

\subsubsection*{Step 3}
By corollary \ref{BanachBundles-cor-homotopicI}, $\mathcal{V}$ and $\mathcal{W}$ are homotopic and hence there exists a finite dimensional subbundle $\mathcal{M}$ of $\pi^\ast\mathcal{F}$ such that $\iota^\ast_0\mathcal{M}=\mathcal{V}$ and $\iota^\ast_1\mathcal{M}=\mathcal{W}$. Moreover, $\mathcal{M}$ is trivial since its restriction to $X\times\{0\}$ is trivial. Consider the bundle morphism $\pi^\ast L:\pi^\ast\mathcal{E}\rightarrow\pi^\ast\mathcal{F}$. Our next aim is to extend $\mathcal{M}$ to a larger bundle over $X\times I$ which is transversal to $\im(\pi^\ast L)$ over $X\times I$.\\
Consider $P\pi^\ast L:\pi^\ast\mathcal{E}\rightarrow\mathcal{M}'$, where $P:\pi^\ast\mathcal{F}\rightarrow\mathcal{M}'$ denotes a projection onto a complement $\mathcal{M}'$ of $\mathcal{M}$ in $\pi^\ast\mathcal{F}$. $P\pi^\ast L$ is a Fredholm morphism and by theorem \ref{FredMor-theorem-transversalexistence}, we can find a finite dimensional subbundle $\mathcal{M}''$ of $\mathcal{M}'$ which is transversal to $\im(P\pi^\ast L:\pi^\ast\mathcal{E}\rightarrow\mathcal{M}')$ over $X\times I$. Taking the direct sum of $\mathcal{M}''$ and $\mathcal{M}$ we finally obtain a finite dimensional subbundle of $\pi^\ast\mathcal{F}$ that is transversal to $\im(\pi^\ast L:\pi^\ast\mathcal{E}\rightarrow\pi^\ast\mathcal{F})$ and contains $\mathcal{M}$ as a subbundle.

\subsubsection*{Step 4}
Using that $\iota^\ast_0=\iota^\ast_1:K(X\times I)\rightarrow K(X)$, \eqref{BanachBundles-cor-prenaturality} and once again the first step of our proof, we now finally obtain

\begin{align*}
[E(L,\mathcal{V})]-[\mathcal{V}]&=[E(L,\iota^\ast_0(\mathcal{M}\oplus\mathcal{M}''))]-[\iota^\ast_0(\mathcal{M}\oplus\mathcal{M}'')]\\
&=[E(\iota^\ast_0\pi^\ast L,\iota^\ast_0(\mathcal{M}\oplus\mathcal{M}''))]-[\iota^\ast_0(\mathcal{M}\oplus\mathcal{M}'')]\\
&=\iota^\ast_0([E(\pi^\ast L,\mathcal{M}\oplus\mathcal{M}'')]-[\mathcal{M}\oplus\mathcal{M}''])\\
&=\iota^\ast_1([E(\pi^\ast L,\mathcal{M}\oplus\mathcal{M}'')]-[\mathcal{M}\oplus\mathcal{M}''])\\
&=[E(\iota^\ast_1\pi^\ast L,\iota^\ast_1(\mathcal{M}\oplus\mathcal{M}''))]-[\iota^\ast_1(\mathcal{M}\oplus\mathcal{M}'')]\\
&=[E(L,\iota^\ast_1(\mathcal{M}\oplus\mathcal{M}''))]-[\iota^\ast_1(\mathcal{M}\oplus\mathcal{M}'')]\\
&=[E(L,\mathcal{W})]-[\mathcal{W}].
\end{align*}
\end{proof}

Because of the theorems \ref{FredMor-theorem-transversalexistence} and \ref{FredMor-theorem-welldefined} we now can finally define the index bundle as follows:

\begin{defi}\label{FredMor-defi-indbund}
Let $L\in\mathfrak{F}(\mathcal{E},\mathcal{F})$ be a Fredholm morphism acting between the Banach bundles $\mathcal{E}$ and $\mathcal{F}$ over the compact base $X$. We call the element

\begin{align*}
\ind(L)=[E(L,\mathcal{V})]-[\mathcal{V}]\in K(X)
\end{align*}
the index bundle of $L$, where $\mathcal{V}\subset\mathcal{F}$ is any finite dimensional subbundle such that \eqref{FredMor-align-transversal} holds.
\end{defi}


\section{Main properties}

In this section we discuss the main properties of the index bundle for Fredholm morphisms between Banach bundles. Since most of them are immediate consequences of the definition, we leave almost all proofs to the reader.\\
If not otherwise stated, we assume throughout that $X$ is a compact topological space, $\mathcal{E},\mathcal{F}$ are Banach bundles over $X$ and $L\in\mathfrak{F}(\mathcal{E},\mathcal{F})$ is a Fredholm morphism.

\begin{lemma}[Normalisation]\label{FredMor-lemma-normalization}
Let $L\in\mathfrak{F}_0(\mathcal{E},\mathcal{F})$ be a bundle isomorphism. Then 

\begin{align*}
\ind(L)=0\in K(X).
\end{align*}
\end{lemma}

\begin{lemma}[Naturality]\label{FredMor-lemma-naturality}
Let $f:Y\rightarrow X$ be continuous, where $Y$ is compact. Then

\begin{align*}
\ind(f^{\ast}L)=f^{\ast}\ind(L)\in K(Y).
\end{align*}
\end{lemma}
Note that we obtain as immediate consequence of the foregoing lemma that $\ind(f^\ast L)=\ind(g^\ast L)\in K(Y)$ if $f\simeq g:Y\rightarrow X$.  

\begin{lemma}[Homotopy invariance property]\label{FredMor-lemma-homotopy}
Let $\mathcal{E}$ and $\mathcal{F}$ be Banach bundles over $X\times I$ and let $L\in\mathfrak{F}(\mathcal{E},\mathcal{F})$ be a Fredholm morphism. Then

\begin{align*}
\ind(i^\ast_0L)=\ind(i^\ast_1L)\in K(X).
\end{align*}
\end{lemma}

\begin{cor}[Invariance under compact perturbations]\label{FredMor-cor-compact}
Let $K\in\mathcal{L}(\mathcal{E},\mathcal{F})$ be a compact operator in every fibre. Then

\begin{align*}
\ind(L+K)=\ind(L)\in K(X).
\end{align*}
\end{cor}

\begin{lemma}[Direct sum property]\label{FredMor-lemma-sum}
Let $M\in\mathfrak{F}(\tilde{\mathcal{E}},\tilde{\mathcal{F}})$ be a further Fredholm morphism between Banach bundles $\tilde{\mathcal{E}}$ and $\tilde{\mathcal{F}}$ over $X$. Then 

\begin{align*}
\ind(L\oplus M)=\ind(L)+\ind(M)\in K(X).
\end{align*}
\end{lemma}

Although the proof of the following important property is also quite elementary, we include it for the convenience of the reader.

\begin{lemma}[Logarithmic property]\label{FredMor-lemma-log}
Let $\mathcal{G}$ be a Banach bundle over $X$ and let $M\in\mathfrak{F}(\mathcal{F},\mathcal{G})$ be a further Fredholm morphism. Then

\begin{align*}
\ind(M\circ L)=\ind(M)+\ind(L)\in K(X).
\end{align*}

\end{lemma}

\begin{proof}
Let $\mathcal{W}\subset\mathcal{G}$ be a finite dimensional subbundle which is transversal to $\im(M\circ L)$. Then $\mathcal{W}$ is transversal to $\im(M)$ as well and hence $E(M,\mathcal{W})$ is defined. Our first aim is to show that $E(M,\mathcal{W})$ is transversal to $\im(L)$.\\
In order to do so, let $u\in\mathcal{F}_\lambda$ for some $\lambda\in X$. Then $M_\lambda u\in\mathcal{G}_\lambda$ and hence we can find $w_0\in\mathcal\im(M_\lambda L_\lambda)$ and $w_1\in\mathcal{W}_\lambda$ such that $M_\lambda u=w_0+w_1$. Now we choose $u_0\in\im(L_\lambda)$ such that $M_\lambda u_0=w_0$ and set $u_1=u-u_0\in\mathcal{F}_\lambda$. Then $M_\lambda u_1=w_1\in\mathcal{W}_\lambda$ and hence $u=u_0+u_1$ where $u_0\in\im(L_\lambda)$ and $u_1\in M^{-1}_\lambda(\mathcal{W}_\lambda)$. Thus $\im(L_\lambda)+M^{-1}_\lambda(\mathcal{W}_\lambda)=\mathcal{F}_\lambda$, $\lambda\in X$, which proves that $E(M,\mathcal{W})$ is transversal to $\im(L)$.\\
Next we observe that

\begin{align*}
E(M\circ L,\mathcal{W})_\lambda&=\{u\in\mathcal{E}_\lambda:M_\lambda L_\lambda u\in\mathcal{W}_\lambda\}=\{u\in\mathcal{E}_\lambda:L_\lambda u\in M^{-1}_\lambda(\mathcal{W}_\lambda)\}\\
&=\{u\in\mathcal{E}_\lambda:L_\lambda u\in E(M,\mathcal{W})_\lambda\}=E(L,E(M,\mathcal{W}))_\lambda,\quad\lambda\in X,
\end{align*}
and hence $E(M\circ L,\mathcal{W})=E(L,E(M,\mathcal{W}))$. We finally obtain

\begin{align*}
\ind(M\circ L)&=[E(M\circ L,\mathcal{W})]-[\mathcal{W}]\\
&=[E(L,E(M,\mathcal{W}))]-[E(M,\mathcal{W})]+[E(M,\mathcal{W})]-[\mathcal{W}]\\
&=\ind(L)+\ind(M).
\end{align*}
\end{proof}

The following reduction property is in particular not available in this general form in the classical definition of the index bundle for families of operators acting between fixed Banach spaces. Note that $\mathcal{V}$ is not assumed to be of finite dimension here.

\begin{lemma}\label{FredMor-lemma-reduction}
Let $\mathcal{V}$ be a direct subbundle of $\mathcal{F}$ which is transversal to $\im(L)$ in the sense of \eqref{FredMor-align-transversal} and such that $L^{-1}_\lambda(\mathcal{V}_\lambda)$ is a complemented subspace of $\mathcal{E}_\lambda$ for all $\lambda\in X$. Then $E(L,\mathcal{V})$ is a direct subbundle of $\mathcal{E}$, $\tilde{L}:=L\mid_{E(L,\mathcal{V})}$ defines an element of $\mathfrak{F}(E(L,\mathcal{V}),\mathcal{V})$ and

\begin{align*}
\ind L=\ind\tilde L\in K(X).
\end{align*}

\end{lemma}

\begin{proof}
It is clear that $E(L,\mathcal{V})$ can be defined as before under the given assumptions on $\mathcal{V}$. Moreover, since $E(L,\mathcal{V})$ is a direct subbundle of $\mathcal{E}$ and $\mathcal{V}$ is a direct subbundle of $\mathcal{F}$, the restriction of $L$ is a bundle morphism. Hence in order to show the second assertion we just have to prove that $\tilde{L}$ is a Fredholm operator in each fibre which follows immediately from $\ker(L_\lambda)=\ker(\tilde L_\lambda)$ and $\im(\tilde L_\lambda)=\im(L_\lambda)\cap\mathcal{V}_\lambda$ for all $\lambda\in X$.\\
To prove the final assertion, let $\mathcal{W}\subset\mathcal{V}$ be a finite dimensional subbundle which is transversal to the image of $\tilde L$. Then $\mathcal{W}$ is transversal to the image of $L$ as well. Moreover, since $\mathcal{W}\subset\mathcal{V}$, we deduce that $E(L,\mathcal{W})=E(\tilde{L},\mathcal{W})$ and the restrictions of $L$ and $\tilde{L}$ to this bundle coincide. Hence

\begin{align*}
\ind\tilde L=[E(\tilde L,\mathcal{W})]-[\mathcal{W}]=[E(L,\mathcal{W})]-[\mathcal{W}]=\ind L.
\end{align*}

\end{proof}

Finally, we want to show that in the case $\mathcal{E}=\mathcal{F}=\Theta(E)$, our definition of $\ind(L)$ is just the classical one as described in section \ref{BanachBundles}.\\
Let $L:X\rightarrow\mathcal{BF}(E)$ be a norm continuous family of bounded Fredholm operators. Let $E_1\subset E$ be a finite codimensional closed subspace such that $\ker L_\lambda\cap E_1=\{0\}$ for all $\lambda\in X$ and set $\mathcal{W}=\im(I-P)\subset\Theta(E)$, where $P$ denotes a projection onto the subbundle $\im(L\mid_{X\times E_1})$ of $\Theta(E)$. Then the index bundle in the classical sense is defined by

\begin{align*}
[\Theta(E/E_1)]-[\mathcal{W}]\in K(X).
\end{align*}
Note that by construction 

\begin{align}\label{FredMor-align-AJdirect}
\im(L_\lambda\mid_{E_1})\oplus\mathcal{W}_\lambda=E_\lambda
\end{align}
which implies that $\im(L_\lambda)+\mathcal{W}_\lambda=E_\lambda$, $\lambda\in X$. Hence we can use $\mathcal{W}$ in order to build the index bundle according to our definition \ref{FredMor-defi-indbund} and obtain

\begin{align*}
\ind L=[E(L,\mathcal{W})]-[\mathcal{W}]\in K(X).
\end{align*}
Now, since $E(L,\mathcal{W})\subset \Theta(E)$ by definition, we have a well defined map

\begin{align*}
F:E(L,\mathcal{W})\rightarrow \Theta(E/E_1),\quad (\lambda,u)\mapsto (\lambda,[u])
\end{align*}
and in order to show that both definitions of the index bundle coincide we want to prove that $F$ is a bundle isomorphism.\\
Note at first that $F$ is obviously continuous and hence it suffices to show that each $F_\lambda$, $\lambda\in X$, is bijective. If $F_\lambda(u)=[u]=0$, then $u\in E_1$ because $E_1$ is closed. Hence $L_\lambda u\in\im(L_\lambda\mid_{E_1})$ and since $L_\lambda u\in\mathcal{W}_\lambda$ by definition of $E(L,\mathcal{W})$, we infer that $L_\lambda u=0$ by \eqref{FredMor-align-AJdirect}. But using that $u\in E_1$ once again, we conclude that $u=0$ and hence $F_\lambda$ is injective. Moreover, we have

\begin{align*}
\dim\mathcal{W}_\lambda&=\dim\coker(L_\lambda\mid_{E_1})=-\ind(L_\lambda\mid_{E_1})=\dim E/E_1-\ind L_\lambda
\end{align*} 
and comparing this equality with \eqref{FredMor-lemma-dimE} yields the assertion.


\section{The kernel of the index map}
As already mentioned in the introduction, the Atiyah-J\"anich theorem states that the index map \eqref{indmap} is an isomorphism for separable Hilbert spaces $H$ and it is proved by using the exactness of \eqref{exactsequence} and Kuiper's theorem. In the case of Banach spaces $E$, the group $GL(E)$ needs no longer be contractible. Moreover, to the best knowledge of the author it is unknown if the index map is surjective in general. However, at least the sequence

\begin{align*}
1\rightarrow[X,GL(E)]\xrightarrow{\iota}[X,\mathcal{BF}(E)]\xrightarrow{\ind}K(X)
\end{align*} 
is still exact (cf. \cite[Theorem 2.1]{Banach Bundles}), that is, the kernel of the index map is given by the set of all families which can be deformed to a family of invertible operators.\\
The aim of this section is to show that this result also holds for Fredholm morphisms between Banach bundles, which is a consequence of the following proposition.

\begin{prop}\label{Bif-prop-parametrix}
Let $X$ be a compact topological space, let $\mathcal{E},\mathcal{F}$ be Banach bundles over $X$ and let $L\in\mathfrak{F}(\mathcal{E},\mathcal{F})$ be a Fredholm morphism. Then $\ind(L)=0\in K(X)$ if and only if there exist a bundle isomorphism $M\in GL(\mathcal{E},\mathcal{F})$ and a fibrewise compact morphism $K\in\mathcal{L}(\mathcal{E},\mathcal{F})$ such that $L=M+K$.
\end{prop}

\begin{proof}
Note at first that $\ind L=0$ if $M$ and $K$ exist because of lemma \ref{FredMor-lemma-normalization} and corollary \ref{FredMor-cor-compact}.\\
Before we treat the remaining assertion we want to show that under the assumption $\ind(L)=0$ we can find a finite dimensional trivial subbundle $\mathcal{V}\subset \mathcal{F}$ which is transversal to the image of $L$ such that the bundle $E(L,\mathcal{V})$ is trivial.\\
Let $\mathcal{V}\subset \mathcal{F}$ be any trivial finite dimensional subbundle transversal to the image of $L$. Since $\ind L=[E(L,\mathcal{V})]-[\mathcal{V}]=0\in K(X)$ we can find a trivial vector bundle $\mathcal{W}$ over $X$ such that 

\begin{align}\label{Bif-align-trivial}
E(L,\mathcal{V})\oplus\mathcal{W}\cong \mathcal{V}\oplus\mathcal{W}
\end{align}
and hence $E(L,\mathcal{V})\oplus\mathcal{W}$ is a trivial bundle. Now we choose a trivial subbundle $\mathcal{V}'\subset \mathcal{F}$ such that 

\begin{align*}
\dim\mathcal{V}'=\dim\mathcal{W},\quad \mathcal{V}_\lambda\cap \mathcal{V}'_\lambda=\{0\},\quad \lambda\in X,
\end{align*}
and consider the bundle $E(L,\mathcal{V}\oplus \mathcal{V}')$. Let $E(L,\mathcal{V})^\perp$ denote a complement of $E(L,\mathcal{V})$ in $E(L,\mathcal{V}\oplus \mathcal{V}')$ and let $P_{\mathcal{V}'}:\mathcal{V}\oplus \mathcal{V}'\rightarrow \mathcal{V}\oplus \mathcal{V}'$ be the projection onto $\mathcal{V}'$ along $\mathcal{V}$. Then the bundle map 

\begin{align*}
P_{\mathcal{V}'}L\mid_{E(L,\mathcal{V})^\perp}:E(L,\mathcal{V})^\perp\rightarrow\mathcal{V}'
\end{align*}
is an isomorphism. Indeed, if $P_{\mathcal{V}',\lambda}L_\lambda u=0$ for some $u\in E(L,\mathcal{V})^\perp_\lambda$, we infer $L_\lambda u\in \mathcal{V}_\lambda$ which implies $u\in E(L,\mathcal{V})_\lambda$ and hence $u=0$. Moreover, since $\dim E(L,\mathcal{V}\oplus \mathcal{V}')=\dim \mathcal{V}+\dim \mathcal{V}'$ and $\dim E(L,\mathcal{V})=\dim \mathcal{V}$ by \eqref{FredMor-lemma-dimE} we deduce that $\dim E(L,\mathcal{V})^\perp=\dim \mathcal{V}'$.\\
We conclude

\begin{align*}
E(L,\mathcal{V}\oplus \mathcal{V}')\cong E(L,\mathcal{V})\oplus E(L,\mathcal{V})^\perp\cong E(L,\mathcal{V})\oplus\mathcal{V}'\cong E(L,\mathcal{V})\oplus\mathcal{W},
\end{align*} 
where the last isomorphism exists because $\mathcal{W}$ is trivial and of the same dimension than $\mathcal{V}'$. Hence $E(L,\mathcal{V}\oplus \mathcal{V}')$ is a trivial bundle by \eqref{Bif-align-trivial}.\\
Now we begin the proof of the remaining assertion. Let $\mathcal{V}\subset \mathcal{F}$ be a finite dimensional trivial subbundle which is transversal to the image of $L$ and such that $E(L,\mathcal{V})$ is a trivial bundle. Since $\dim E(L,\mathcal{V})=\dim \mathcal{V}$ by \eqref{FredMor-lemma-dimE}, we can find an isomorphism $a:E(L,\mathcal{V})\rightarrow\mathcal{V}$. Let $P\in\mathcal{L}(\mathcal{E})$ and $Q\in\mathcal{L}(\mathcal{F})$ be projections such that $\im P_\lambda=E(L,\mathcal{V})_\lambda$ and $\im Q_\lambda=\mathcal{V}_\lambda$ , $\lambda\in X$. Then we can define a bundle morphism $K\in\mathcal{L}(\mathcal{E},\mathcal{F})$ by

\begin{align*}
K_\lambda=a_\lambda\circ P_\lambda-Q_\lambda\circ L_\lambda,\quad\lambda\in X,
\end{align*}
which consists of finite rank operators. Now we consider

\begin{align*}
L_\lambda+K_\lambda=(I_{\mathcal{F}_\lambda}-Q_\lambda)L_\lambda+a_\lambda P_\lambda
\end{align*}
which is a Fredholm morphism of index $0$. If $L_\lambda u+K_\lambda u=0$ for some $u\in \mathcal{E}_\lambda$, we infer $(I-Q_\lambda)L_\lambda u=a_\lambda P_\lambda u=0$ because $a_\lambda P_\lambda u\in \mathcal{V}_\lambda$ and $Q_\lambda$ is a projection onto $\mathcal{V}_\lambda$. Now from $(I-Q_\lambda)L_\lambda u=0$ we obtain $L_\lambda u\in \mathcal{V}_\lambda$ while $a_\lambda P_\lambda u=0$ gives $P_\lambda u=0$ because $a_\lambda$ is an isomorphism. But, since $P_\lambda$ is a projection onto $E(L,\mathcal{V})_\lambda=L^{-1}_\lambda(\mathcal{V}_\lambda)$, we conclude from $L_\lambda u\in\mathcal{V}_\lambda$ that $P_\lambda u=u$ and so finally $u=0$.\\
Hence each $L_\lambda+K_\lambda$ is an injective Fredholm operator of index $0$ and so an isomorphism. 
\end{proof}

We now immediately obtain the following important corollary.

\begin{cor}
$\ind L=0\in K(X)$ if and only if $L$ is homotopic to a bundle isomorphism. That is, there exists a bundle morphism $H:\pi^\ast\mathcal{E}\rightarrow\pi^\ast\mathcal{F}$ such that $\iota^\ast_0H=L$ and $\iota^\ast_1H$ is invertible. 
\end{cor}


\section{An application}

Let $n\in\mathbb{N}$, let $\mathcal{V}\subset\Theta(\mathbb{K}^{2n})$ be an $n$-dimensional subbundle of the $2n$-dimensional product bundle over the compact topological space $X$ and let $1\leq p\leq\infty$. We consider the family of unbounded operators on $L^p(I,\mathbb{K}^{n})$ defined by

\begin{align*}
L_{B_x}u=u',\quad \mathcal{D}(L_{B_x})=\{u\in W^{1,p}(I,\mathbb{K}^{n}):(x,u(0),u(1))\in\mathcal{V}_x\}.
\end{align*}
We leave it as an exercise for the reader to show that each $L_{B_x}$ is an unbounded Fredholm operator of (numerical) index $0$ on $L^p(I,\mathbb{K}^{n})$. Let us point out that boundary value problems of the type we consider here have been studied for a long time (cf. \cite{Conti} for a review).\\
In the following we not only consider $W^{1,p}(I,\mathbb{K}^n)$ as a subspace of $L^p(I,\mathbb{K}^n)$ but also as a Banach space in its own right with respect to the usual Sobolev norm. Now the aim of this section is to prove the following assertion.

\begin{prop}
There exists a continuous family of compact operators 

\begin{align*}
K:X\rightarrow\mathcal{L}(W^{1,p}(I,\mathbb{K}^n),L^p(I,\mathbb{K}^n))
\end{align*}
such that $L_{B_x}+K_x$ is invertible for all $x\in X$, if and only if $\mathcal{V}$ is stably trivial.
\end{prop}

We let $P\in\mathcal{L}(\Theta(\mathbb{K}^{2n}))$ denote a projection onto $\mathcal{V}$ and observe at first that we have a surjective bundle map

\begin{align*}
\Theta(W^{1,p}(I,\mathbb{K}^n))\rightarrow\im(I-P),\quad (x,u)\mapsto (I-P_x)(x,u(0),u(1)).
\end{align*} 
Hence the kernel $\mathfrak{D}$ of this morphism is a direct subbundle of $\Theta(W^{1,p}(I,\mathbb{K}^n))$ whose total space consists of the set of all domains $\mathcal{D}(L_{B_x})$, $x\in X$. Moreover, the family $L_B$ defines a Fredholm morphism $\tilde{L}:\mathfrak{D}\rightarrow \Theta(L^p(I,\mathbb{K}^n))$.\\
Next we show that a family of compact operators as in the assertion of our proposition exists if and only if $\ind\tilde{L}=0\in K(X)$. We assume at first that the family of compact operators $K$ exists. Then we obtain by restriction a bundle morphism 

\begin{align*}
\tilde{K}:\mathfrak{D}\rightarrow \Theta(L^p(I,\mathbb{K}^n))
\end{align*}
which is compact in every fibre and such that $\tilde{L}_x+\tilde{K}_x$ maps $\mathfrak{D}_x=\mathcal{D}(L_{B_x})$ to $L^p(I,\mathbb{K}^n)$ bijectively. Hence $\ind\tilde{L}=0$ by lemma  \ref{FredMor-lemma-normalization} and corollary \ref{FredMor-cor-compact}. If, on the other hand, the index bundle of $\tilde{L}$ is trivial, then by proposition \ref{Bif-prop-parametrix} there exists a bundle morphism $\tilde{K}:\mathfrak{D}\rightarrow \Theta(L^p(I,\mathbb{K}^n))$ which is compact in every fibre and such that $\tilde{L}+\tilde{K}$ is an isomorphism. Now we extend $\tilde{K}$ to $\Theta(W^{1,p}(I,\mathbb{K}^n))$ trivially and obtain a family of compact operators $K:X\rightarrow\mathcal{L}(W^{1,p}(I,\mathbb{K}^n),L^p(I,\mathbb{K}^n))$.\\
Our final aim is to compute $\ind\tilde{L}$.  The space $L^p(I,\mathbb{K}^{n})$ can be decomposed as direct sum of the $n$-dimensional space $Y_1$ of constant functions and

\begin{align*}
Y_2=\left\{y\in L^p(I,\mathbb{K}^{n}):\int^1_0{y(s)\,ds}=0\right\}.
\end{align*} 
If $v\in Y_2$, then $u(t):=\int^t_0{v(s)\, ds}$, $t\in I$, defines an element of $\mathcal{D}(L_{B_x})$ such that $L_{B_x}u=v$. Hence $L_{B_x}(\mathcal{D}(L_{B_x}))\supset Y_2$ and we conclude

\begin{align*}
\im L_{B_x}+Y_1=L^p(I,\mathbb{K}^{n}),\quad x\in X.
\end{align*}
Since the total space of $E(\tilde{L},\Theta(Y_1))$ is given by

\begin{align*}
\{(x,w)\in& X\times W^{1,p}(I,\mathbb{K}^{n}):w\in\mathcal{D}(L_{B_x}), w'\in Y_1\}\\
&=\{(x,w)\in X\times W^{1,p}(I,\mathbb{K}^{n}):w\in\mathcal{D}(L_{B_x}), w'\equiv const. \}\\
&=\{(x,w)\in X\times W^{1,p}(I,\mathbb{K}^{n}):w(t)=(1-t)a+tb,\, (x,a,b)\in\mathcal{V}_x\},
\end{align*}
we see that $E(\tilde{L},\Theta(Y_1))$ is isomorphic to $\mathcal{V}$. Now we finally infer

\begin{align*}
\ind\tilde{L}=[\mathcal{V}]-[\Theta(\mathbb{K}^n)]\in K(X),
\end{align*}
which is trivial if and only if $\mathcal{V}$ is stably trivial because of the assumption $\dim\mathcal{V}=n$.

\bigskip

\begin{flushleft}

Nils~Waterstraat,\\
Dipartimento di Scienze Matematiche, Politecnico di Torino\\
Corso Duca Degli Abruzzi 24, 10129 Torino, Italy\\
e-mail: \texttt{waterstraat@daad-alumni.de}\\[2ex]
\end{flushleft}

\end{document}